\documentclass[10pt,a4paper]{article}

\usepackage[leqno]{amsmath}
\usepackage{amssymb,amsthm,upref,amscd}
\usepackage[T1]{fontenc}
\usepackage{times}
\usepackage{cite}
\usepackage{amsfonts}
\usepackage[colorinlistoftodos]{todonotes}
\usepackage{setspace}

\usepackage{color}
\usepackage{hyperref}
\usepackage{graphicx}
\usepackage{wrapfig}
\usepackage{setspace}

\usepackage[usenames,dvipsnames]{pstricks}
\usepackage{epsfig}

\newtheorem{thm}{Theorem}[section]
\newtheorem{lem}[thm]{Lemma}
\newtheorem{prop}[thm]{Proposition}

\newtheorem{example}[thm]{Example}

\theoremstyle{definition}

\theoremstyle{remark}
\newtheorem*{notation}{Notation}
\theoremstyle{remark}
\newtheorem{remark}[thm]{Remark}
\numberwithin{equation}{section}

\newcommand{\N}{{\mathbb N}
}

\newcommand{\R}{{\mathbb R}}

\definecolor{blue}{rgb}{0,0,1}

\newcommand{\NN}{{\mathbb N}}
\newcommand{\RR}{{\mathbb R}}

\newcommand{\gop}{\gamma\left(\frac{u^2+|\nabla u|^2}{2}\right)}

\newcommand{\di}{\displaystyle}

\newcommand{\ri}{\rightarrow}
\newcommand{\mce}{{\mathcal E}}

\newcommand{\intom}{\int_\Omega}

\newcommand{\beq[1]}{\begin{equation}\label{eq:#1}}
\newcommand{\eeq}{\end{equation}}

\begin{document}

\title{Nonhomogeneous quasilinear elliptic problems: linear and sublinear cases}

\author{Louis Jeanjean\footnote{Laboratoire de Math\'ematiques
UMR 6623, Universit\'e de Bourgogne Franche-Comt\'e,
16 route de Gray,
25030 Besan\c con Cedex, France ; {\tt	louis.jeanjean@univ-fcomte.fr}} \ $\&$ Vicen\c{t}iu D. R\u{a}dulescu\footnote{Faculty of Applied Mathematics, AGH University of Science and Technology, 30-059 Krakow, Poland \& Department of Mathematics, University of Craiova, 200585 Craiova, Romania; {\tt radulescu@inf.ucv.ro}}}

\maketitle

\begin{abstract}
We are concerned with a class of second order quasilinear elliptic equations driven by a nonhomogeneous differential operator introduced by C.A. Stuart \cite{stuart} and whose study is motivated by models in Nonlinear Optics. We establish sufficient conditions for the existence of at least one or two non-negative solutions. Our analysis considers the cases when the reaction has either a sublinear or a linear growth. In the sublinear case, we also prove a nonexistence property. The proofs combine energy estimates and variational methods.

\smallskip
\noindent{\bf Keywords}: mountain pass solution, Palais-Smale condition, non-homogeneous differential operator.

\smallskip\noindent
{\bf 2010 Mathematics Subject Classification:} 35J60, 35J62, 58E05.
\end{abstract}

\section{Introduction}

Let
$\Omega$ be a bounded open subset of $\RR^N$.
Consider the linear Dirichlet problem
\begin{equation}\label{pro00}
\left\{  \begin{array}{ll}
    -\Delta u+ u=\lambda u+h &\mbox{in $\Omega$} \\
    u= 0 &\mbox{on $\partial\Omega$},
    \end{array}\right.
\end{equation}
where $h\in L^2(\Omega)$ is a given function and $\lambda$ is a real parameter. Then the following results are true:\\
(i) if $h\equiv 0$ then problem \eqref{pro00} has a positive solution (which is unique up to a multiplicative constant) if and only if $\lambda=1+\lambda_1$, where $\lambda_1$ is the lowest eigenvalue of the Laplace operator $(-\Delta)$ in $H^1_0(\Omega)$;\\ (ii) if $h \gneqq 0$ then problem \eqref{pro00} has a positive solution if and only if $\lambda<1+\lambda_1$. Moreover, this solution is unique. We refer to H.~Brezis \cite[Chapter 9]{brezis} for more details. \smallskip

These classical results have been extended recently by C.A.~Stuart \cite{stuart}, provided that the left-hand side in problem \eqref{pro00} is replaced with the nonlinear differential operator
$$Su:= -\mbox{div}\,\left[\gop\nabla u\right]+\gop u,$$
where
 $\gamma:[0,\infty)\ri\RR$ is a positive continuous function. Namely, he considered the problem
\begin{equation}\label{pro0}
\left\{  \begin{array}{ll}
 \displaystyle   -\mbox{div}\,\left[\gop\nabla u\right]+\gop u=  \lambda u +h &\mbox{in $\Omega$} \\
    u= 0 &\mbox{on $\partial\Omega$}.
    \end{array}\right.
\end{equation}

In \cite{stuart}, under the assumption that $h \in L^2(\Omega)$, $h \geq 0$, sufficient conditions such that problem \eqref{pro0} admits two non-negative weak  solutions are established. \smallskip

A motivation for replacing $- \Delta $ by the operator $S$ stems from C.A.~Stuart and H.-S.~Zhou's works \cite{stuzho1,stuzhocv,stuzhosiam,stuzho2} in relationship with guided traveling waves propagating through a  self-focusing dielectric.
This problem is central to the study of  transverse electric
field modes (TE-modes) propagating in an axisymmetric dielectric such
as an optical fiber.
The mathematical analysis of such phenomena in a nonlinear dielectric medium is part of the study of special solutions of Maxwell's  equations coupled with a nonlinear constitutive relation between the electric field and the electric displacement field.
The main reason in this relationship is that a TE-mode is a solution of Maxwell's equations in which the electric field is an
axisymmetric, monochromatic traveling wave which is everywhere transverse to the
direction of propagation. The analysis developed by C.A.~Stuart and H.-S.~Zhou's includes the case of guided
TM-modes propagating through a self-focusing anisotropic dielectric. These are special solutions
of Maxwell's equations with a nonlinear constitutive relation of a type commonly
used in nonlinear optics when treating the propagation of waves in a cylindrical wave-guide. \smallskip

When $\gamma(t)\equiv 1$ (or more generally when $\gamma(t)$ is a positive constant) the operator $S$ reduces to the Laplace operator. The case where $\gamma(t)$ is non constant corresponds to a quasilinear setting. \smallskip

In \cite{stuart}, C.A.~Stuart introduced the following assumptions:
\begin{itemize}
\item \textbf{(s1)} $\gamma$ is non-increasing on $[0,\infty)$ and $\gamma(t) \underset{t \rightarrow \infty}{\longrightarrow} \gamma(\infty)>0$.
\item \textbf{(s2)} Setting $\Gamma(t):= \int_0^t \gamma(s) ds,$  there exists $\rho > 0,\;$ such that  for all $ (t,s) \in [0,\infty)^2,$ $$\Gamma(t^2) \geq \Gamma(s^2) + 2s\gamma(s^2)(t - s) + \rho(t - s)^2.$$
\item \textbf{(s3)} Setting $K(t) = \Gamma(t) - \Gamma'(t)t$, we have $ \lim_{t \to \infty}K(t) < \infty.$
\end{itemize}

In \cite[p. 329]{stuart}, examples of functions $\gamma$ satisfying (s1)-(s3) are given.

Under the above assumptions the main result in  \cite{stuart} is concerned with the case  $\gamma(\infty)<\gamma(0)$. The existence of two non-negative weak solutions is established for all $\gamma(\infty)+\lambda_1\gamma(\infty) < \lambda < \gamma(0 )+\lambda_1\gamma(\infty)$, assuming that $h \geq 0$ is sufficiently small. Actually, defining
$$\Psi (u) := \int_{\Omega} \Gamma \Big(\frac{u^2 + |\nabla u|^2}{2}\Big)\, dx - \frac{\lambda}{2} \int_{\Omega} u^2 \, dx - \int_{\Omega} hu \, dx,$$
then a first solution is obtained as a local minima of $\Psi (u)$. The second solution corresponds to a mountain pass level for $\Psi(u)$. See \cite[Theorem 1.1]{stuart} for a precise statement. \smallskip

\begin{remark}
Setting $F(z,p) = \Gamma(\frac{1}{2}[z^2 + |p|^2])$ for $z \in \RR$ and $p \in \RR^N$, it is well-known, see for example \cite[Chapter 10]{Gi-Tr}, that the ellipticity of \eqref{pro0} is equivalent to the convexity of $F(z,p)$ with respect to $p$ for all $(z,p)$. Also it is easily checked that this condition corresponds to the convexity of $g(t) = \Gamma(t^2)$ on $(0, \infty)$. Now, as shown in \cite{stuart}, the stronger hypothesis $(s2)$ ensures the uniform ellipticity of \eqref{pro0}.
\end{remark}

The mathematical treatment of \eqref{pro0} is characterized by two main difficulties. Since the function $\gamma(t)$ is bounded between two positive constants, the right-hand side of \eqref{pro0} has a, somehow, linear growth as well as the reaction term $\lambda u$. It is then expected that to prove the existence of a bounded Palais-Smale (or Cerami) sequence at the mountain pass level for $\Psi(u)$ may be challenging.  This is reminiscent of what happen in problems of type \eqref{pro0} with a linear operator (such as the Laplacian) on the left-hand side and a nonlinearity which is asymptotically linear on the right-hand side. Then, one need to understand precisely the interaction between the nonlinearity and the spectrum of the linear operator in order to prove the existence of {\it a priori} bounds on the Palais-Smale sequences. In this direction we refer, for instance, to \cite{BaBeFo,GiLuRa,JeTa,MaSo}.  In the present paper, this difficulty is reinforced by the nonlinear character of the quasilinear operator and also by the fact that its nonhomogeneous character do not permit to benefit from certain classical techniques as, for example, the one presented in \cite{TaTeUl}. We also note that these features of the quasilinear operator make the study of its action on weakly convergent sequences non standard.  \smallskip

The aim of the present paper is twofold. First, we extend the results in \cite{stuart} to the case when the reaction has a nonlinear growth.
We consider the problem
\begin{equation}\label{pro1}
\left\{  \begin{array}{ll}
\displaystyle    -\mbox{div}\,\left[\gop\nabla u\right]+\gop u=  f(u)+h &\mbox{in $\Omega$} \\
    u= 0 &\mbox{on $\partial\Omega$},
    \end{array}\right.
\end{equation}
where $f$ is a given continuous function and $h\in L^2(\Omega)$ is non-negative.

We distinguish between the cases when $f$ has either a sublinear decay or a linear growth at infinity. Secondly, we optimize the assumptions on the quasilinear part, namely on the function $\gamma : [0, \infty) \to \RR$. In particular, condition $(s3)$ will be no more needed and $(s2)$ is replaced by the assumption that the function $\Gamma(t^2)$ is strictly convex on $[0, \infty)$.

\section{Main results}
We first study the case where the right-hand side in problem \eqref{pro1} has a sublinear growth (in a prescribed sense). Next, we extend the results obtained in Stuart \cite{stuart}, provided that $f$ is no longer linear, but it has a linear growth.
In both settings we establish sufficient conditions on $f$ that guarantee the existence of non-negative solutions for problem \eqref{pro1}.

\subsection{Sublinear growth case}

In this case we write $f(t)$ as $f(t)= \nu g(t)$ for some $\nu >0$ and assume that $g: \mathbb{R} \rightarrow \mathbb{R}$ is a continuous function. The following assumptions will be used to state our results:

\begin{itemize}
\item \textbf{(g1)} $g(t) =o(t)$ as $t \to + \infty$,
\item \textbf{(g2)} $\exists t_0 \in [0, \infty)$ such that $G(t_0) := \displaystyle\int_{0}^{t_0} g(s)\, ds > 0,$
\item \textbf{(g3)} $\exists C>0, \forall t \in \mathbb{R}, \lvert g(t) \rvert \leq C\lvert t \rvert$.
\end{itemize}

Note that since we are only interested in non-negative solutions we can assume, without lack of generality, that
 $g(t) =0$ for all $t \in (- \infty, 0].$

\begin{example}\label{ex2} The following functions satisfy hypotheses (g1)--(g3):

(i) $g(t)=\sin (at)$ for $a\in \RR\setminus\{0\}$;

(ii) $g(t)=|t|^\alpha/(1+|t|^\beta)$ for  $1\leq\alpha <\beta+1$;

(iii) $g(t)=\min\{|t|^\alpha,|t|^\beta\}$  for  $0<\alpha <1<\beta$;

(iv) $g(t)=\log(1+|t|)$;

(v) $g(t)=\exp \{(\log(1+|t|))^\alpha\}-1$  for $\alpha\in (0,1)$;

(vi) $g(t)=\exp \{\log(1+|t|)/\log\log (2+|t|)\}-1$.
\end{example}

Concerning the operator $S$, we assume :
\begin{itemize}
\item \textbf{(q1)} The function $\gamma : [0, \infty) \to \RR$  is continuous and there exist some constants $0 < \gamma_{\text{min}} \leq \gamma_{\text{max}}$ such that
$\forall t \in [0, \infty),$ $ \gamma_{\text{min}} \leq \gamma(t) \leq \gamma_{\text{max}}.$
\item \textbf{(q2)}  The function $ t \mapsto  \Gamma(t^2)$ is convex on $[0, \infty).$
\end{itemize}

\begin{thm}\label{thm1} Assume that the conditions $(g1)$ and $(q1)-(q2)$ are fulfilled and that $h \in L^2(\Omega)$, $h \geq 0$. We have the following properties for problem \eqref{pro1}, where $f(s) = \nu g(s)$,
\begin{itemize}
\item[(i)] Let $h \gneqq 0$. Then problem  \eqref{pro1} admits a non-negative nontrivial solution all $\nu > 0$.
\item[(ii)] Let $h \equiv 0$, then we have,
\begin{enumerate}
\item[(a)] If we assume $(g1)-(g2)$, there exists a $\nu_1 > 0$ such that problem \eqref{pro1} admits a non-negative nontrivial solution if $\nu > \nu_1.$
\item[(b)] If we assume $(g1)-(g3)$, there exists a  $\nu_0 > 0$ such that problem \eqref{pro1}
 does not have any non trivial solution if $0 < \nu < \nu_0.$
\end{enumerate}
\end{itemize}
\end{thm}

 In particular, Theorem \ref{thm1} asserts that if $h \equiv 0$ and if the nonlinear term $g$ in the right-hand side of problem  \eqref{pro1} has a sublinear decay at $ +\infty$ and at most a linear growth near the origin, then the parameter $\nu >0$ must be large enough in order to guarantee the existence of solutions. This corresponds to high perturbations of the reaction. \smallskip

Returning to the semilinear case, which corresponds to $\gamma(0)=\gamma(\infty)$, Theorem \ref{thm1} asserts that, provided that $g$ satisfies assumptions $(g1)-(g3)$, then the following Dirichlet problem
\begin{center}

$\left\{
\begin{tabular}{lr}
$-\Delta u + u = \nu g(u) $& in $\Omega$\\
$\qquad u = 0$&on $\partial\Omega$
\end{tabular}
\right.$
\end{center}

\noindent does not have nontrivial solutions for $\nu > 0$ small enough, but admits non-negative solutions for $\nu$ sufficiently large. For example, Theorem \ref{thm1} shows that the problem

\begin{center}

$\left\{
\begin{tabular}{lr}
$-\Delta u + u = \nu\text{log}(1+\lvert u \rvert) $& in $\Omega$\\
$\qquad u = 0$&on $\partial\Omega,$
\end{tabular}
\right.$

\end{center}

\noindent has only the trivial solution if $\nu>0$ is small enough and that nontrivial solutions do exist as soon as $\nu>0$ is sufficiently large. By contrast, if $h \gneqq 0$ (and $h \in L^2\left(\Omega\right)$), then the nonlinear problem

\begin{center}

$\left\{
\begin{tabular}{lr}
$-\Delta u + u = \nu\text{log}(1+\lvert u \rvert) + h$& in $\Omega$\\
$\qquad u = 0$&on $\partial\Omega,$
\end{tabular}
\right.$

\end{center}

\noindent has a non-negative solution for all $\nu > 0$.

\subsection{Linear growth case}

Here we extend the main results obtained in C.A.~Stuart \cite{stuart}, which are concerned with the linear case $f(u)= \lambda u $ in problem \eqref{pro1}.  \medskip

Concerning the operator $S$ we need, with respect to the sublinear growth case, to strengthen our assumptions $(q1)$ and $(q2)$ by requiring the following conditions.

\begin{itemize}
\item
 \textbf{(q3)} The function $\gamma : [0, \infty) \to \RR$  is continuous and there exist some constants $0 < \gamma_{\text{min}} \leq \gamma_{\text{max}}$ such that
$\forall t \in [0, \infty),$ $ \gamma_{\text{min}} \leq \gamma(t) \leq \gamma_{\text{max}}.$ In addition, the exists $\gamma(\infty) >0$ such that $\gamma(s) \underset{s \rightarrow \infty}{\longrightarrow} \gamma(\infty).$
\item
 \textbf{(q4)} The function $t \mapsto \Gamma(t^2)$ is strictly convex on $[0, \infty).$

\end{itemize}

\begin{remark}\label{equivalent1}
The condition $(q4)$ is rather standard in the literature dealing with related quasilinear problems \cite{CaGoSi, SiCaSiGo, Ke}. Note that it is equivalent to require that the function $t \mapsto t \gamma(\frac{t^2}{2})$ is strictly increasing on $[0, \infty)$.
\end{remark}

We consider a large class of nonlinearities with linear growth and satisfying natural hypotheses, in strong relationship with the first eigenvalue $\lambda_1$ of the Laplace operator $(-\Delta)$ in $H^1_0(\Omega)$. Depending of the results seek we shall ask the continuous function $f:[0, \infty) \ri\RR$ to satisfy some of the following conditions:

\begin{itemize}
\item \textbf{(f1)}  $\displaystyle\limsup_{t \rightarrow 0}\frac{f(t)}{t} < \gamma(0) + \gamma_{\text{min}}\, \lambda_1,$
\item \textbf{(f2)} $\displaystyle\limsup_{t \rightarrow + \infty}\frac{f(t)}{t^p} =0,$ for some $p \in (1, 2^*-1)$,
\item \textbf{(f3)} $\displaystyle\liminf_{t \rightarrow +\infty}\frac{f(t)}{t} > \gamma(\infty)(1 + \lambda_1)$,
\item \textbf{(f4)} $\displaystyle\limsup_{t \rightarrow +\infty}\frac{f(t)}{t} < +\infty.$
\end{itemize}

\begin{thm}\label{thm2} Assume that the conditions $(q3)$ hold and let $h \in L^2(\Omega)$, $h \geq 0$. Then the following properties hold.
\begin{itemize}
\item[(i)] If $(f1)-(f2)$ and $(q2)$ holds, then for $||h||_2$ sufficiently small, problem \eqref{pro1} admits a least one non-negative solution. In addition, this solution is nontrivial if $h \not \equiv 0$.
\item[(ii)] If $(f1)-(f3)-(f4)$ and $(q4)$ hold, then for $||h||_2$ sufficiently small, problem \eqref{pro1} admits at least two nontrivial non-negative solutions if $h \not\equiv 0$  and at least a nontrivial non-negative solution if $h \equiv 0$.
\end{itemize}
\end{thm}

Clearly, the main challenge is to establish Theorem \ref{thm2} (ii), namely the existence of two non-negative solutions. As in \cite{stuart}, a first solution will correspond to a local minima of an associated functional and the second one will lie at its mountain pass level. It is well known that a mountain pass geometry implies the existence of a Palais-Smale sequence $(u_n)_{n \in \N}$ at the mountain pass level. To obtain a critical point it then suffices to show that this Palais-Smale sequence is, up to a subsequence, converging. In that direction a first difficulty is to show that $(u_n)_{n \in \N}$ is bounded. To derive the existence of this bounded Palais-Smale sequence, the paper \cite{stuart} relies on a previous work of the author \cite{St}, which can be seen as an alternative version of some results in \cite{Gh}. It relies on the so-called notion of {\it localizing the Palais-Smale sequence}.
\smallskip

In order to overcome this difficulty, we follow in the present work a different strategy. On the one hand, we introduce a different choice of the functional in order to ensure, from the beginning, that a non-negative critical point will be obtained. On the other hand, we make use of the approach developed in \cite{Je} to obtain a bounded Palais-Smale sequence. With respect to \cite{stuart} this approach permits to work with $(q3)$ instead of $(s1)$ and without having to require $(s3)$. \smallskip

Having obtained a bounded Palais-Smale sequence of the mountain pass level it remains to prove its convergence, up to a subsequence. It is on that point that requiring $(s2)$ instead of  $(q4)$ proved to be necessary in the approach developed in \cite{stuart}. Here, combining recent developments \cite{SiCaSiGo, CaGoSi} as well as classical results due to Brezis and Lieb \cite{BrLi}, we manage to show that only hypothesis $(q4)$ is needed. \smallskip

This being said, we acknowledge that our proofs strongly rely, at several key points, on elements first established in \cite{stuart}. \smallskip

Lastly, we point out that our assumption $(g3)$, which implies that the function $\gamma(t)$ lies between two positive constants, forbid to consider certain classes of operators.  For example the $p$-Laplacian or more generally $\Phi$-Laplacian operators as defined in \cite{SiCaSiGo}, see also \cite{Ke}. However, to the best of our knowledge, our problem does not fall into a well-defined category. In particular the fact that the term $u^2$ is present inside the function $\gamma(t)$ seems to be new, apart from \cite{St} obviously, and likely some elements developed in this paper will prove useful to consider versions of problem \eqref{pro1} set on the whole space $\R^N$.

\subsubsection*{Further comments about main hypotheses} A basic example of potentials that fulfill all hypotheses $(q1)$--$(q4)$ in the statement of Theorems \ref{thm1} and \ref{thm2} is given by
$$\Gamma (t)=At+B[(1+t)^{p/2}-1],$$ where $A$, $B$ are positive numbers and $1<p<2$. In this case, the energy functional associated to problem \eqref{pro1} is a {\it double-phase} variational integral, according to the terminology of P.~Marcellini and G.~Mingione. This corresponds to quantities of the type
\begin{equation}\label{model}\int_\Omega (|\nabla u|^p+a(x)|\nabla u|^2)dx\ \  (1<p<2)\end{equation}
in the expression of the energy defined by \eqref{energy}, where $a(x)$ is a nonnegative potential. Functionals of this type have been studied for the first time by P.~Marcellini \cite{marce1,marce2}, in relationship with patterns arising in nonlinear elasticity, see J.~Ball \cite{ball1, ball2}. More precisely, Marcellini studied variational integrals of the type $\int_\Omega f(x, \nabla u)$dx, where $f=f(x,\xi)$ is a function with {\it unbalanced growth} satisfying
$$c_1\,|\xi|^p\leq |f(x,\xi)|\leq c_2\,(1+|\xi|^2)\quad\mbox{for all}\ (x,u)\in\Omega\times\RR^N,$$
for some positive constants $c_1$ and $c_2$. The study of non-autonomous functionals characterized by the fact that the energy density changes its ellipticity and growth properties according to the point has been continued in a series of remarkable papers by G.~Mingione {\it et al.} \cite{beck,2Bar-Col-Min, 9Col-Min}. These contributions are in relationship with the papers of V.~Zhikov \cite{23Zhikov,24Zhikov}, which describe the
behavior of phenomena arising in nonlinear
elasticity.
In fact, Zhikov intended to provide models for strongly anisotropic materials in the context of homogenisation and Lavrentiev-type phenomena.
In particular, he initiated the qualitative analysis of the energy functional defined in \eqref{model},
where the modulating coefficient $a(x)$ dictates the geometry of the composite made of
two differential materials, with hardening exponents $p$ and $2$, respectively.

  \medskip
\begin{notation}
For $N \geq 1, 1\leq p<\infty,$ $L^p(\Omega)$ is the usual Lebesgue space with norm
$||u||_p^p := \int_{\Omega}|u|^p \,dx.$
The Sobolev space $H^1_0(\Omega)$ is endowed with its equivalent norm $||u||^2 = \int_{\Omega} |\nabla u|^2 \, dx.$
We denote by $'\rightarrow'$, respectively by $'\rightharpoonup'$, the strong convergence, respectively the  weak convergence in corresponding space. Finally, for a function $ u \in H_0^1(\Omega)$, we recall that $u_+ = \max \{u, 0\}$ and $u_- = - \min \{u, 0 \}.$
\end{notation}

\section{Preliminaries}

We recall that  $u_0 \in H^1_0\left(\Omega\right)$ is a weak solution to problem \eqref{pro1}  if
$$\displaystyle\int_{\Omega} \gamma\left(\frac{u_0^2 + \lvert \nabla u_0 \rvert^2}{2}\right)\left(u_0v + \nabla u_0 . \nabla v\right) \, dx = \displaystyle\int_{\Omega}f(u_0)v\, dx + \displaystyle\int_{\Omega} hv \, dx \quad \forall v \in H_0^1\left(\Omega\right).$$

\noindent Thereafter, we shall often consider the following auxiliary problem
\begin{equation}\label{pro2}
\left\{  \begin{array}{ll}
  \displaystyle  -\mbox{div}\,\left[\gop\nabla u\right]+\gop u=  f(u_+)+h &\mbox{in $\Omega$} \\
    u= 0 &\mbox{on $\partial\Omega$}.
    \end{array}\right.
\end{equation}
\begin{lem}\label{newproblem}
Assume that $\gamma \geq \gamma_0 >0$ on $[0, \infty)$ and $h \geq 0$. Then any solution $u \in H^1_0(\Omega)$ to problem \eqref{pro2} is non-negative. In particular,
$u \in H^1_0(\Omega)$  is solution to problem \eqref{pro1}.
\end{lem}
\begin{proof}
 Let $u \in H^1_0(\Omega)$ be a solution to problem \eqref{pro2}. Multiplying \eqref{pro2} by $u_-$ and integrating we obtain
$$\displaystyle\int_{\Omega} \gamma\left(\frac{u^2 + \lvert \nabla u \rvert^2}{2}\right)\left(u u_- + \nabla u  \nabla u_-\right) \, dx = \displaystyle\int_{\Omega}f(u_+)u_-\, dx + \displaystyle\int_{\Omega} h u_- \, dx.$$
\noindent Therefore, using that $\displaystyle\int_{\Omega} h u_-\, dx \geq 0$ (since $h \geq 0$) and $\displaystyle\int_{\Omega} f(u_+) u_-\, dx = 0$, we obtain that
$$-\displaystyle\int_{\Omega} \gamma\left(\frac{u^2 + \lvert \nabla u \rvert^2}{2}\right)\left(\lvert u_- \rvert^2 + \lvert\nabla u_-\right \rvert^2) \, dx \geq 0.$$
It follows that $u_- = 0$, hence $u \geq 0$ is a solution to problem \eqref{pro1}.
\end{proof}
Associated to \eqref{pro2} we introduce the functional
\begin{equation}\label{energy}\mathcal{E}(u) :=  \displaystyle\int_{\Omega} \Gamma\left(\frac{u^2 + \lvert \nabla u \rvert^2}{2}\right) \, dx - \displaystyle\int_{\Omega}F(u_+)\, dx - \displaystyle\int_{\Omega} hu\, dx \quad u \in H_0^1\left(\Omega\right)\end{equation}
where $F(t):=\displaystyle\int_0^t f(s)\, ds$. For later use we also introduce, for $v \in H^1_0(\Omega)$,
$$ \Phi(v):= \int_{\Omega} \Gamma \Big( \frac{v^2 + |\nabla v |^2}{2}\Big) \, dx.$$
\begin{lem}\label{semi}
Assume that $f : [0, \infty) \to \RR$ is continuous and satisfies $(f2)$. Then the following properties hold,
\begin{enumerate}
\item[(i)] Assume that $\gamma \in L^{\infty}(\Omega)$. Then  $\mathcal{E} \in C^1(H^1_0(\Omega))$ with
\begin{equation}\label{derivee}
 \mathcal{E}'(u)v = \displaystyle\int_{\Omega} \gamma\left(\frac{u^2 + \lvert \nabla u \rvert^2}{2}\right)\left(uv + \nabla u  \nabla v\right) \; dx - \displaystyle\int_{\Omega}f(u_+)v\; dx - \displaystyle\int_{\Omega} hv \, dx.
\end{equation}
\item[(ii)] Moreover, if $(q2)$ is satisfied, then $\mathcal{E}$ is weakly sequentially lower semicontinuous.
\end{enumerate}
\end{lem}

\begin{proof}
For Part (i) it suffices to use \cite[C.1 Theorem]{Struwe}. Here we only need to use the fact that $\gamma$ is bounded from above.

For Part (ii) we follow closely \cite[Proposition 3.1]{stuart}. Let $w_n \rightharpoonup w$ weakly in $H^1_0(\Omega)$. Because of the compactness of the Sobolev embeddings of $H^1_0(\Omega)$ into $L^q(\Omega)$ for any $q \in [2, 2^*)$ it is enough to prove that
$$\Phi(w) \leq \liminf_{n \to \infty}\Phi(w_n).$$
For $u \in H^1_0(\Omega)$ and $v \in H^1_0(\Omega)$, let $y = (u, \nabla u)$ and $z = (v, \nabla v)$ so that $y \cdot z = uv + \nabla u \cdot \nabla v$. Setting $g(s):= \Gamma(s^2)$ and using $(q2)$ we have
\begin{equation} \label{x}
\begin{aligned}
\Phi(u) - \Phi(v)  = \int_{\Omega} g\Big(\frac{|y|}{\sqrt{2}}\Big) - g\Big(\frac{|z|}{\sqrt{2}}\Big) \, dx
 & \geq \int_{\Omega} g'\Big(\frac{|z|}{\sqrt{2}}\Big)
 \Big(\frac{|y|- |z|}{\sqrt{2}} \Big) \, dx\\
& = \int_{\Omega} \gamma \Big(\frac{|z|^2}{2}\Big)
\sqrt{2}|z| \Big(\frac{|y|- |z|}{\sqrt{2}} \Big) \, dx\\
& =  \int_{\Omega} \gamma \Big(\frac{|z|^2}{2}\Big)
|z| (|y|- |z|) \, dx.
\end{aligned}
\end{equation}
Also, from Part (i),
\begin{equation}\label{103}
\begin{aligned}
\Phi'(v)(u - v) & = \int_{\Omega} \gamma \Big(\frac{|z|^2}{2}\Big)  \big( v(u-v) + \nabla v \cdot \nabla (u -v) \big) \, dx \\
& = \int_{\Omega} \gamma \Big(\frac{|z|^2}{2}\Big) z \cdot (y-z) \,dx.
\end{aligned}
\end{equation}
From \eqref{x} and \eqref{103}, we see that, for all $u \in H^1_0(\Omega)$ and $v \in H^1_0(\Omega)$,
\begin{equation}\label{111}
\Phi(u) - \Phi(v) - \Phi'(v)(u- v)
 \geq \int_{\Omega} \gamma \Big(\frac{|z|^2}{2}\Big) \big(|z|(|y| - |z|) - z \cdot(y-z) \big) \, dx
  \geq 0.
\end{equation}
Choosing $u = w_n$ and $v = w$ in \eqref{111}, this yields
$$ \Phi(w) \leq \Phi(w_n) - \Phi'(w)(w_n - w)$$
where $\Phi'(w)(w_n -w) \to 0$ since $\Phi'(w)$ is in the dual of $H^1_0(\Omega)$ and $w_n \rightharpoonup w$ weakly in $H^1_0(\Omega).$
\end{proof}

To prove the convergence of some Palais-Smale sequences we shall make use of the following three technical results. The first one is a particular case of \cite[Lemma 6]{La} that we recall here for completeness.
\begin{prop}\label{LEM1}
Let $X$ be a finite dimensional real Hilbert space with norm $|\cdot|$ and scalar product $\langle\cdot, \cdot \rangle$. Let $\beta : X \to X$ be a continuous function which is strictly monotone, that is,
\begin{equation}\label{D01}
\langle \beta(\xi) - \beta(\bar{\xi}), \xi-\bar{\xi} \rangle>0, \quad \mbox{for every} \, \,  \xi ,\bar{\xi} \in X \mbox{ with } \xi \neq \bar{\xi}.
\end{equation}
Let $(\xi_n)_{n \in \N} \subset X$ and $\xi \in X$ be such that
$$\lim_{n \to \infty}\langle\beta(\xi_n) - \beta(\xi), \xi_n - \xi\rangle =0.$$
Then $(\xi_n)_{n \in \N}$ converges to $\xi$ in $X$.
\end{prop}

Our next result is a version of already existing related properties, see in particular \cite[Proposition 2.5]{SiCaSiGo}. We thank the authors of \cite{SiCaSiGo} to have indicated to us such results.

\begin{lem}\label{LEM2}
Assume that $(q3)-(q4)$ hold. The assumptions of Proposition \ref{LEM1} are satisfied with the choice $X = \R^{N+1}$ equipped with the standard euclidian scalar product $\langle \cdot, \cdot \rangle $ and $ \beta : \R^{N+1} \to \R^{N+1}$ given by
$$\beta(\xi) = \gamma\Big(\frac{|\xi|^2}{2}\Big)\xi.$$
\end{lem}
\begin{proof}
Let us split the proof into two parts. In the first part we consider $\xi, \bar{\xi} \in \R^{N+1}$ satisfying $|\xi| = |\bar{\xi}|$. We have that
$$\big\langle \gamma\Big(\frac{|\xi|^2}{2}\Big)\xi - \gamma \Big(\frac{|\bar{\xi}|^2}{2}\Big)\bar{\xi}, \xi - \bar{\xi} \big\rangle = \gamma \Big(\frac{|\xi|^2}{2}\Big)|\xi - \bar{\xi}|^2$$
and thus if $\xi \neq \bar{\xi}$ this quantity is strictly positive. Next, we assume that $|\xi| \neq |\bar{\xi}|$. By the Cauchy-Schwartz inequality  we get
\begin{equation*}
\begin{aligned}
\big\langle \gamma \Big(\frac{|\xi|^2}{2}\Big)\xi -  \gamma \Big(\frac{|\bar{\xi}|^2}{2}\Big)\bar{\xi}, \xi - \bar{\xi}\big\rangle& \geq  \gamma\Big(\frac{|\xi|^2}{2}\Big)|\xi|(|\xi| - |\bar{\xi}|) +
\gamma \Big(\frac{|\bar{\xi}|^2}{2}\Big)|\bar{\xi}|(|\bar{\xi}| - |\xi|)   \\
 & = \gamma \Big(\frac{|\xi|^2}{2}\Big)|\xi| - \gamma \Big(\frac{|\bar{\xi}|^2}{2}\Big)|\bar{\xi}|( |\xi| - |\bar{\xi}|).
\end{aligned}
\end{equation*}
Recording that the function $s \to  \displaystyle \gamma \Big(\frac{s^2}{2}\Big)s$ is strictly increasing, see Remark \ref{equivalent1}, the conclusion follows.
\end{proof}

\begin{lem}\label{LEM3}
Assume that $(g3)-(g4)$ hold. If $(u_n)_{n \in \N} \subset H^1_0(\Omega)$ is such that
$\Phi'(u_n) (u_n-u) \to 0$ then $\nabla u_n(x) \to \nabla u(x)$ a.e. in $\Omega$.
\end{lem}
\begin{proof}
Since $u_n \rightharpoonup u$ weakly in $H_0^1(\Omega)$ we also have that $\Phi'(u)(u_n- u) \to 0$ and thus
\begin{equation}\label{D1}
\Big[ \Phi'(u_n) - \Phi'(u)\Big](u_n-u) \to 0.
\end{equation}
Introducing the notation, for each fixed $x \in \Omega$,
$$ y_n(x) = (u_n(x), \nabla u_n(x)) \in \R^{N+1} \quad \mbox{and} \quad y_n(x) = (u(x), \nabla u(x)) \in \R^{N+1},$$
we see from \eqref{derivee}, taking \eqref{D1} into account, that
$$ \Big[ \Phi'(u_n) - \Phi'(u)\Big](u_n-u) = \int_{\Omega} \big\langle \gamma \Big(\frac{|y_n|^2}{2}\Big)y_n - \gamma \Big(\frac{|y|^2}{2}\Big)y, y_n- y \big\rangle \, dx \to 0.$$
Here $\langle \cdot, \cdot \rangle$ is the standard euclidian scalar product on $\R^{N+1}$. It follows that
$$ \big\langle \gamma \Big(\frac{|y_n|^2}{2}\Big)y_n - \gamma \Big(\frac{|y|^2}{2}\Big)y, y_n- y \big\rangle  \to 0 \quad \mbox{in} \, L^1(\Omega)$$
and thus that
$$ \langle \gamma \Big(\frac{|y_n|^2}{2}\Big)y_n - \gamma \Big(\frac{|y|^2}{2} \Big)y, y_n- y \big\rangle  \to 0 \quad \mbox{a.e.} \, x \in \Omega.$$
At this point, in view of Proposition \ref{LEM1} and Lemma \ref{LEM2}, we deduce that $y_n(x) \to y(x)$ in $\R^{N+1}$ for almost every $x \in \Omega$. In particular $\nabla u_n(x) \to \nabla u(x)$, a.e. in $\Omega$.
\end{proof}

Finally, in order to prove Theorem \ref{thm2}, we shall also need the following result which says that for a wide class of functionals, having a mountain-pass  geometry, almost every functional in this class has a bounded Palais-Smale sequence at the mountain pass level. This is \cite[Theorem 1.1]{Je}.

\begin{thm}\label{Jeanjean} Let $X$ be a Banach space equipped with the norm $\lVert . \rVert$ and let $J \subset \mathbb{R}_+$ be an interval. We consider a family $(I_{\mu})_{\mu \in J}$ of $C^1$-functionals on $X$ of the form
$$I_{\mu}(u) = A(u) - \mu B(u) \quad \forall \mu \in J,$$
where $B(u) \geq 0, \forall u \in X$ and such either $A(u) \to + \infty$ or $B(u) \to + \infty$ as $||u|| \to \infty$. \smallskip

We assume there are two points $(v_1,v_2)$ in $X$ such that setting
$$\Sigma = \left\{ \sigma \in C\left([0,1],X\right) / \sigma(0) = v_1,\; \sigma(1) = v_2 \right\}$$
there holds, for all $\mu \in J$
$$c_{\mu} := \inf_{\sigma \in \Sigma} \max_{t \in [0,1]} I_{\mu}\left(\sigma(t)\right) > \max\left\{I_{\mu}(v_1),I_{\mu}(v_2)\right\}.$$
Then, for almost every $\mu \in J$, there is a sequence $(v_n)_{n \in \mathbb{N}} \subset X$ such that
\begin{enumerate}
\item[(i)]  $ (v_n)_{n \in \mathbb{N}} \mbox{ is bounded }$;
\item[(ii)] $ I_{\mu}(v_n) \underset{n \rightarrow \infty}{\longrightarrow} c_{\mu}$;
\item[(iii)] $ I'_{\mu}(v_n)  \underset{n \rightarrow \infty}{\longrightarrow} 0  \mbox{ in the dual } X^{-1} \mbox{ of } X. $
\end{enumerate}
\end{thm}

\section{Sublinear growth case}
In this section we are concerned with the proof of Theorem \ref{thm1}.

\begin{lem}\label{coercive}
Assume that $(g1)$ and $(q1)$ hold. Then the functional $\mathcal{E}$ is coercive.
\end{lem}

\begin{proof}
Assumption (g1) implies that
$$G(t)=o(t^2)\qquad\mbox{as $t\ri\pm\infty$}.$$
Thus, for $\nu>0$ fixed, there exists $A>0$ such that for all $t\in\RR$
$$ \nu G(t)\leq\frac{\gamma_{min}}{2}\, t^2+A.$$
Since $\Gamma (t)\geq \gamma_{min}t$, we have for all $u\in H^1_0(\Omega)$
$$\begin{array}{ll}\di {\mathcal E} (u)&\di\geq \frac{\gamma_{min}}{2}\intom(u^2+|\nabla u|^2)dx-\frac{\gamma_{min}}{2}\intom u^2dx-A\,|\Omega|\\
&=\di \frac{\gamma_{min}}{2}\, \|u\|^2-A\,|\Omega|\,,\end{array}$$
hence $\mathcal{E}$  is coercive and bounded from below.
\end{proof}

We now can prove the first part of the theorem. \smallskip

\noindent {\bf{Proof of Theorem \ref{thm1}(i).}}
According to Lemmas \ref{semi} and \ref{coercive}, the functional  $\mathcal{E}$ admits a global minimum $u \in H_0^1\left(\Omega\right)$, which is thus a solution to
\eqref{pro2}. It corresponds, as shown in Lemma \ref{newproblem}  to a non-negative solution of \eqref{pro1}. In addition, $u$ is nontrivial since $h \not \equiv 0$.
$\Box$
\smallskip

From now on, we consider the specific case $h \equiv 0$.
\begin{lem}\label{negatif} If $h \equiv 0$, there exists $\nu_1 > 0$ such that for all $\nu > \nu_1$,
$$\underset{u \in H_0^1\left(\Omega\right)}\inf \mathcal{E}(u) < 0.$$
\end{lem}

\begin{proof}
This result follows by using our hypothesis (g2). We arbitrarily fix a compact set $K \subset \Omega$. By Tietze's extension theorem, there exists a map $w \in H_0^1\left(\Omega\right) \cap C\left(\overline{\Omega}\right)$ such that $w \equiv t_0$ in $K$ and $\lvert w \rvert \leq  t_0 $ in $\Omega$. Therefore
$$\displaystyle\int_{\Omega} G(w)\;dx = G(t_0)\lvert K \rvert + \displaystyle\int_{\Omega \backslash K} G(w)\;dx \geq G(t_0)\lvert K \rvert - \max_{s \in [0,t_0]}|G(s)| \, \lvert \Omega \backslash K \rvert.$$
\noindent This highlights that if $\lvert K \rvert$ approaches $\lvert \Omega \rvert$ (hence, if $\lvert \Omega \backslash K \rvert$ is small) then $\int_{\Omega} G(w)\;dx > 0$. On the other hand, we have by (q1) that $\Gamma(t) \leq \gamma_{\text{max}}t$ for all $t \geq 0$. It follows that for all $\nu$ sufficiently large we have :
$$\mathcal{E}(w) \leq \frac{\gamma_{\text{max}}}{2}\displaystyle\int_{\Omega}(w^2 + \lvert \nabla w \rvert^2)\; dx - \nu\displaystyle\int_{\Omega}G(w)\; dx < 0$$
and the lemma is proved.
\end{proof}

\begin{lem}\label{nonexistence}
Assuming (g3), problem \eqref{pro1} does not have a nontrivial solution for $\nu >0$ small.
\end{lem}

\begin{proof}
 Let $\nu > 0$ be such that the problem \eqref{pro1} admits a nontrivial solution $u \in H_0^1\left(\Omega\right)$. Using $u$ as a test function  we get that
$$\displaystyle\int_{\Omega} \gamma\left(\frac{u^2 + \lvert \nabla u \rvert^2}{2}\right)\left(u^2 +  \lvert \nabla u \rvert^2\right) \; dx = \left \lvert \nu\displaystyle\int_{\Omega}g(u)u\; dx \right\rvert \leq \nu\displaystyle\int_{\Omega}\lvert g(u) \rvert \lvert u \rvert\; dx.$$
So, according to the assumption (g3), we obtain, for some constant $C >0$,
$$\displaystyle\int_{\Omega} \gamma\left(\frac{u^2 + \lvert \nabla u \rvert^2}{2}\right)\left(u^2 +  \lvert \nabla u \rvert^2\right) \; dx \leq \nu C\lVert u \rVert^2_2.$$
Recalling that $\gamma \geq \gamma_{min} >0$ we then get
$$\gamma_{min}\lVert u \rVert_2^2 \leq \displaystyle\int_{\Omega} \gamma\left(\frac{u^2 + \lvert \nabla u \rvert^2}{2}\right)\left(u^2 +  \lvert \nabla u \rvert^2\right) \, dx.$$
\noindent By these two last inequalities, we deduce that there exists $\nu_0  > 0$ such that for all $0 < \nu < \nu_0$,  problem \eqref{pro1} does not admit any nontrivial solution.
\end{proof}

\begin{proof}[Proof of Theorem \ref{thm1}(ii).]
From Lemma \ref{negatif} we deduce that the global minimum obtained as a consequence of Lemmas \ref{semi} and \ref{coercive} is non trivial if $\nu >0$ is sufficiently large and this proves Part (a). Now Part (b) follows directly from Lemma \ref{nonexistence}.
\end{proof}

\section{Linear growth case}

\noindent In this section we are deal with the proof of Theorem \ref{thm2}, which corresponds to a linear growth of the nonlinearity in problem \eqref{pro1}. Roughly speaking, assuming that $||h||_2$ is sufficiently small, we shall prove that the functional $\mathcal{E}(u)$ possesses a mountain pass geometry. A first solution, nontrivial if $h \not \equiv 0$, is then obtained as a local minima. Finding a second solution, corresponding to the mountain pass level, is more involved because of the lack of {\it a priori} bounds on the Palais-Smale sequences.  At this step we shall make use of the strategy developed by the first author in \cite{Je}, which relies on Theorem \ref{Jeanjean}.

\subsection{Existence of a local minima for $\mathcal{E}(u)$}

\begin{lem}\label{localminstructure}
Assume that $(f1)-(f2)$ and $(q2)-(q3)$ hold. If $||h||_2$ is sufficiently small, then there exist  $r >0$ and $\alpha >0$ such that for all $u\in H^1_0(\Omega)$ with $\|u\|=r$, $\mce (u)\geq\alpha$.
\end{lem}

\begin{proof}
We write
\begin{equation}\label{5.1}
\Gamma \left(\frac{1}{2}(u^2 + |\nabla u|^2) \right) = \int_{0}^{\frac{1}{2}u^2 + \frac{1}{2}|\nabla u|^2} \gamma(s) \, ds = \int_{0}^{\frac{1}{2}u^2} \gamma(s) \, ds + \int_{\frac{1}{2}u^2}^{\frac{1}{2}u^2 + \frac{1}{2}|\nabla u|^2} \gamma(s) \, ds.
\end{equation}
Using (q3) we get
\begin{equation}\label{5.2}
\int_{\frac{1}{2}u^2}^{\frac{1}{2}u^2 + \frac{1}{2}|\nabla u|^2} \gamma(s) \, ds \geq \frac{1}{2} \gamma_{\text{min}}|\nabla u|^2.
\end{equation}
Let $h(s) = \gamma(0) - \gamma(s)$. By the continuity of $\gamma$ at $0$ it follows that
$$ \frac{H(s)}{s} \to 0 \, \mbox{ as } \, s \to 0  \quad \mbox{ where } \quad H(s) = \int_0^s h(t) dt.$$
We can write
\begin{equation}\label{5.3}
 \int_{0}^{\frac{1}{2}u^2} \gamma(s) \, ds = \frac{1}{2} \gamma(0) u^2 - \int_{0}^{\frac{1}{2}u^2} h(s) \, ds = \frac{1}{2}\gamma(0) u^2 - H(\frac{u^2}{2}).
\end{equation}
Now, defining $G \in C([0, \infty))$ by $G(0)=0$ and  $G(s) = \displaystyle \frac{2}{s}H\left(\frac{s}{2}\right)$ if $s >0$, it comes that
$ \displaystyle H\left(\frac{u^2}{2}\right) = G(u^2)\frac{u^2}{2}$ and we deduce, using \eqref{5.1}-\eqref{5.3}, that
\begin{equation}\label{5.4}
\mathcal{E}(u) \geq \displaystyle\int_{\Omega}\left(\frac{\gamma(0)u^2 + \gamma_{\text{min}}\lvert \nabla u \rvert^2}{2} - \frac{G(u^2)u^2}{2}\right)\, dx - \displaystyle\int_{\Omega}F(u_+)\, dx - \displaystyle\int_{\Omega}hu \, dx.
\end{equation}
In view of (f1)-(f2) there exist $\delta >0$ and $C(\delta) >0$ such that, for all $t \in [0, \infty)$,
\begin{equation}\label{5.5}
F(u_+) \leq \left(\frac{\gamma(0) + (\gamma_{\text{min}} - \delta)\lambda_1 }{2}\right)u_+^2 + C(\delta)u_+^{p + 1}.
\end{equation}
It follows from \eqref{5.4}-\eqref{5.5} that there exist $\delta >0$ and  $C(\delta) >0$ such that
$$\begin{array}{ll}\di\mce (u) &\di\geq \int_{\Omega}\frac{\gamma(0)u^2 +  (\gamma_{\text{min}} - \delta )\lvert \nabla u \rvert^2}{2}\; dx   + \frac{\delta}{2} \int_{\Omega} |\nabla u|^2 dx  -\frac{\gamma(0)+(\gamma_{\text{min}} - \delta)\lambda_1}{2}\intom u_+^2 \, dx \smallskip\\
&\di -
\frac12\intom G(u^2)u^2dx - C(\delta) \int_{\Omega}u_+^{p+1} \, dx - \int_{\Omega}h u \, dx \smallskip \\
&\di \geq  \frac{\delta}{2} \int_{\Omega} |\nabla u|^2 \, dx -
\frac12\intom G(u^2)u^2 \, dx - C(\delta) \int_{\Omega}u_+^{p+1} dx - \int_{\Omega}h u \, dx.
\end{array}$$

We claim, arguing as in \cite{stuart}, that the term $\intom G(u^2)u^2dx$ is negligible with respect with $\intom |\nabla u|^2dx$, provided that $u\in H^1_0(\Omega)$ is such that $\|u\|=r>0$ is small enough. Indeed, we first notice that, fixing an arbitrary $q \in (1, \frac{2^*}{2}),$
\begin{equation}\label{5.6}
0 \leq \int_{\Omega}  G(u^2)u^2 \, dx \leq \left(\displaystyle\int_{\Omega}  G(u^2)^{q'}\, dx \right)^{\frac{1}{q'}}\left(\displaystyle\int_{\Omega}  u^{2q}\, dx\right)^{\frac{1}{q}} \leq  C \left(\displaystyle\int_{\Omega}  G(u^2)^{q'}\, dx\right)^{\frac{1}{q'}} ||u||^2
\end{equation}
 where $1/q + 1/q' =1$ and for some constant $C>0$. Now since $G$ is a non-negative bounded function on $[0, \infty)$, the map  $u \mapsto G(u^2)^{q'}$ is continuous from $L^2\left(\Omega\right)$ into $L^1\left(\Omega\right)$. Since $G(0)=0$ the claim follows from \eqref{5.6}.

At this point, taking $||u||$ small enough and since $p >1$ we obtain that
$$\mce (u) \di\geq \frac{\delta}{4} \int_{\Omega} |\nabla u|^2 \, dx   - \int_{\Omega}h u \, dx \geq \frac{\delta}{4} \int_{\Omega} |\nabla u|^2 \, dx - ||h||_2 ||u||_2.$$
Finally, taking $||h||_2$ small enough the lemma follows.
\end{proof}

\begin{lem}\label{firstsolution}
Assume that $(f1)-(f2)$, $(q2)-(q3)$ holds and that $||h||_2$ is as in Lemma \ref{localminstructure}. Then problem \eqref{pro1} admits a solution which is a local minimizer of $\mathcal{E}$. In addition, this solution is nontrivial if $h \not \equiv 0.$
\end{lem}

\begin{proof}
We set $$m := \underset{u \in \overline{B(0,r)}}{\text{inf}} \mathcal{E}(u)$$
where $r>0$ is given in Lemma \ref{localminstructure}. Now if  $(v_n)_{n \in \mathbb{N}} \subset \overline{B(0,r)}$ is a minimizing sequence it is obviously bounded and we can assume that $v_n \rightharpoonup v$ in $H_0^1(\Omega)$. By Lemma \ref{semi}(ii) we deduce that $\mathcal{E}(v)=m$. The fact that $v$ is a critical point of $\mathcal{E}(u)$ follows since $m \leq \mathcal{E}(0) = 0 < \alpha.$ Clearly,  $v$ is nontrivial if $h \not \equiv 0$.
\end{proof}

\subsection{A mountain pass geometry }

Combined with Lemma \ref{localminstructure}, the next lemma shows that the functional  $\mathcal{E}$ has a mountain pass geometry around the origin if $||h||_2$ is sufficiently small.

\begin{lem}\label{pointoutside}
Assume that $(f2)-(f3)$ holds and let $\varphi_1>0$ be the first eigenfunction of the Laplace operator in $H^1_0(\Omega)$. Then $\mathcal{E} (t\varphi_1)<0$ for all $t\in\RR$ sufficiently large.
\end{lem}

\begin{proof}
\smallskip
We have, for any $t>0$,
$$\mce (t\varphi_1)=\intom\Gamma\left(\frac{t^2}{2}\, (\varphi_1^2+|\nabla\varphi_1|^2) \right) \, dx-\intom F(t\varphi_1) \,dx - t \intom h \varphi_1 \,dx.$$
Observe that, since $\varphi_1 >0$, we get for any $x \in \Omega$,
$$
\lim_{t\ri\infty}\frac{\Gamma\left(\frac{t^2}{2}\, (\varphi_1(x)^2+|\nabla\varphi_1(x)|^2) \right)}{t^2}=\frac{\gamma(\infty)}{2}[\varphi_1(x)^2+|\nabla\varphi_1(x)|^2].
$$
Also, since $\gamma$ is bounded, we have for all $x \in \Omega$ and for some constant $C>0$,
$$
\lim_{t\ri\infty}\frac{\Gamma\left(\frac{t^2}{2}\, (\varphi_1(x)^2+|\nabla\varphi_1(x)|^2) \right)}{t^2} \leq C [\varphi_1(x)^2+|\nabla\varphi_1(x)|^2].
$$
Thus, by the Lebesgue dominated convergence theorem, we deduce that
\begin{equation}\label{ajout1}
\begin{aligned}
\lim_{t\ri\infty} \frac{1}{t^2}\intom \Gamma\left(\frac{t^2}{2}\, (\varphi_1(x)^2+|\nabla\varphi_1(x)|^2) \right) \, dx & = \frac12\intom\gamma(\infty)[\varphi_1(x)^2+|\nabla\varphi_1(x)|^2] \, dx \\
& = \frac12 \gamma(\infty)(1+ \lambda_1)\intom \varphi_1(x)^2 \, dx.
\end{aligned}
\end{equation}
Next, by $(f3)$, there are $\delta>0$ and $C>0$ such that for all $t>0$
$$F(t)\geq\frac{\gamma(\infty)(1+\lambda_1+\delta)}{2}\,t^2-C.$$
It follows that
\begin{equation}\label{ajout2}
\limsup_{t\ri\infty}\frac{1}{t^2}\intom F(t\varphi_1(x))dx\geq \frac{\gamma(\infty)(1+\lambda_1+\delta)}{2}\intom\varphi_1(x)^2 \,dx.
\end{equation}
From \eqref{ajout1} and \eqref{ajout2} we deduce that
$$\limsup_{t\ri\infty}\frac{\mce (t\varphi_1)}{t^2} \leq -\frac{\delta\gamma(\infty)}{2}\intom\varphi_1(x)^2dx$$
proving the lemma.
\end{proof}

\subsection{Existence of a suitable Palais-Smale sequence and proof of Theorem \ref{thm2}}

In view of Lemmas \ref{localminstructure} and \ref{pointoutside}, the functional $\mathcal{E}$ has a mountain pass geometry. The rest of the paper will be devoted to find a critical point at the mountain pass level and this will end the proof of Theorem \ref{thm2}. Actually we shall exhibit a particular Palais-Smale sequence, at this level, for which it is possible to show its convergence (to a critical point thus). To this aim the strategy,  first presented in \cite{Je},  and which consists in embedding the problem into a family of problems will be put at work. \medskip

We set the continuous functions $f_1 : [0, \infty) \rightarrow \mathbb{R}$ and $f_2 : [0, \infty) \rightarrow \mathbb{R}$ such that, for all $t \in [0, \infty)$,
$$f_1(t) = \text{max}\left(f(t),0\right) \qquad \text{and} \qquad f_2(t) = f_1(t) - f(t).$$
By definition $f_1(t) \geq 0$ for all $t \in [0, \infty)$. Also since $(f3)$  implies that $f(t) \geq 0$ for $t \in [0, \infty)$ sufficiently large, we see that $f_2(t) =0$ for $t$ large and thus there exists a $K < \infty$ such that, for all $ t \in [0, \infty)$
\begin{equation}\label{l1}
F_1(t) := \int_0^{t}f_1(s)\, ds \geq 0  \quad \mbox{and} \quad F_2(t) := \int_0^{t}f_2(s)\, ds \geq - K.
\end{equation}
Obviously we have
$$\displaystyle\int_{\Omega}F(u_+)\, dx = \displaystyle\int_{\Omega}F_1(u_+)\, dx - \displaystyle\int_{\Omega}F_2(u_+)\, dx.$$
Let $J := (1- \alpha, 1 + \alpha)$ for some $\alpha >0$ small. We define the family of $C^1$-functionals $(\mathcal{E}_{\mu})_{\mu \in J}$ on $H_0^1\left(\Omega\right)$ by
$$\mathcal{E}_{\mu}(u) = \displaystyle\int_{\Omega} \Gamma\left(\frac{u^2 + \lvert \nabla u \rvert^2}{2}\right) \; dx - \mu\displaystyle\int_{\Omega}F_1(u_+)\; dx + \displaystyle\int_{\Omega}F_2(u_+)\; dx - \displaystyle\int_{\Omega} hu\; dx.$$
Setting $A$ and $B$ such that :
$$\forall u \in H_0^1\left(\Omega\right),\; A(u) = \displaystyle\int_{\Omega} \Gamma\left(\frac{u^2 + \lvert \nabla u \rvert^2}{2}\right) \; dx   + \displaystyle\int_{\Omega}F_2(u_+)\; dx - \displaystyle\int_{\Omega} hu\; dx$$
$$\text{and}\quad B(u) = \displaystyle\int_{\Omega}F_1(u_+)\; dx,$$
we see that the family of $C^1$-functionals $(\mathcal{E}_{\mu})_{\mu \in J}$ can be written as
$$\forall u \in H_0^1\left(\Omega\right),\; \mathcal{E}_{\mu}(u) = A(u) - \mu B(u).$$
Let us show that this family satisfies the assumptions of Theorem \ref{Jeanjean}. First we see from \eqref{l1} that $B(u) \geq 0$, for all $u \in H_0^1(\Omega)$. Now, using $(q3)$ and again  \eqref{l1}, if follows that, for all $u \in H_0^1(\Omega)$
$$ A(u) \geq \frac{\gamma(\infty)}{2}\lVert u \rVert^2 - \lVert h \rVert_2\lVert u \rVert_2 + K$$
and thus it holds that $A(u) \to + \infty$ as $||u|| \to \infty.$ \smallskip

Finally, inspecting the proofs of Lemmas \ref{localminstructure} and \ref{pointoutside},  we see that the conclusions of these results still hold if we replace $\mathcal{E}$ by $\mathcal{E}_{\mu}$ with $\mu \in J= (1- \alpha, 1 + \alpha)$ for a sufficiently small $\alpha >0$. Thus our family $(\mathcal{E}_{\mu})_{\mu \in J}$ satisfies the assumptions of Theorem \ref{Jeanjean} and we deduce the existence of an increasing sequence $(\mu_n)_{n \in N}$ such that $\mu_n \to 1$ and $\mathcal{E}_{\mu_n}$ has a bounded Palais-Smale sequence at the mountain pass level $c_{\mu_n}$, for any $n \in \NN$. \medskip

The following lemma will be crucial to establish the convergence of the bounded Palais-Smale sequences for $\mathcal{E}_{\mu}$ with $\mu \in J$.

\begin{lem} \label{Brezis-Lieb}
Assume that $(q3)-(q4)$ hold. Let $(u_n)_{n \in \N} \subset H^1_0(\Omega)$, satisfying $u_n \rightharpoonup u$ in $H^1_0(\Omega)$ and $\nabla u_n \to \nabla u$ a.e. on $\Omega$, be such that
\begin{equation}\label{112}
\limsup_{n \to \infty}  \Phi(u_n) \leq \Phi(u).
\end{equation}
Then, up to a subsequence, $u_n \to u$ strongly in $H^1_0(\Omega)$.
\end{lem}
\begin{proof}
We make uses of classical results from \cite{BrLi}. Adopting the notations introduced there, we write $f_n = f + g_n$ with
$$ f_n = \Big( \frac{u_n^2 + |\nabla u_n|^2}{2}\Big)^{\frac{1}{2}} \quad \mbox{and} \quad f = \Big( \frac{u^2 + |\nabla u|^2}{2}\Big)^{\frac{1}{2}}.$$
Note that, passing if necessary to a subsequence still denoted $(u_n)_{n \in \N}$, we have $g_n \to 0$ a.e.  Now setting $j(s) := \Gamma(s^2)$ we know, from $(q3)$, that the function $j$ is continuous, convex on $\R$ with $j(0)=0$. In view of \cite[Example (b)]{BrLi}, where we take $k=2$, we can apply \cite[Theorem 2]{BrLi} to deduce that
$$\int_{\Omega} j(f+g_n) - j(g_n) - j(f) \, dx \to 0.$$
Namely that
$$ \Phi(u_n) - \Phi(u_n -u) - \Phi(u) \to 0.$$
Now, in view of \eqref{112}, we deduce that
\begin{equation}\label{term}
 \Phi(u_n -u) \to 0.
\end{equation}
At this point, recording that $\Gamma(0)=0$ and $\Gamma'(s) = \gamma(s) \geq \gamma_{\text{min}}>0$ and thus that $\Gamma(s) \geq \gamma_{\text{min}} \, s$ we deduce from \eqref{term} that $u_n \to u$ in $H^1_0(\Omega)$.
\end{proof}

\begin{lem} \label{boundedsequence}
Assume that $(f1)-(f2)$ and $(q3)-(q4) $ are fulfilled.
Every bounded Palais-Smale sequence for $\mathcal{E}_{\mu}$ with $\mu \in J$  admits a convergent subsequence.
\end{lem}

\begin{proof}
Let $(u_n)_{n \in \N} \subset H^1_0(\Omega)$ be a bounded Palais-Smale sequence for $\mathcal{E}_{\mu}$. We can assume without restriction that $u_n \rightarrow u$ weakly in $H_0^1(\Omega)$ and $u_n \to u$ in $L^q(\Omega)$, for any $q \in [2, 2^*[$.  Using that $\mathcal{E}_{\mu}'(u_n)(u-u_n) \to 0$ it readily follows using the strong convergence properties that $\Phi'(u_n)(u-u_n) \to 0 $ from which we deduce, see Lemma \ref{LEM3}, that
\begin{equation}\label{110}
\nabla u_n(x) \rightarrow \nabla u(x), \quad \mbox{a.e. in } \Omega.
\end{equation}
Also, from \eqref{111}, where we have set $u =u$ and $v = u_n$, we know that
\begin{equation}\label{ajout3}
\Phi(u) - \Phi(u_n) - \Phi'(u_n)(u- u_n) \geq 0.
\end{equation}
This implies that
\begin{equation}\label{120}
 \limsup_{n \to \infty} \Phi(u_n) \leq \Phi(u).
\end{equation}
At this point, in view of \eqref{110} and \eqref{120}, the conclusion follows from Lemma \ref{Brezis-Lieb}.
\end{proof}

In view of Lemma \ref{boundedsequence} we deduce the existence of a sequence $\left(\mu_n,u_n\right)_{n \in \mathbb{N}} \subset J \times H_0^1\left(\Omega\right)$ such that :
\begin{itemize}
\item $\mu_n \underset{n \rightarrow \infty}{\longrightarrow} 1$ and $(\mu_n)_{n \in \mathbb{N}}$ is increasing.
\item  $u_n \geq 0$, \, $\mathcal{E}_{\mu_n}(u_n) = c_{\mu_n}$ and $\mathcal{E}'_{\mu_n}(u_n) = 0$ in the dual $H_0^{-1}\left(\Omega\right)$.
\end{itemize}

\begin{prop}\label{sequencebounded}
The sequence $(u_n)_{n \in \NN}$ is bounded.
\end{prop}

\begin{proof}
Using $u_n$ as a test function, we get that :
$$ \mathcal{E}'_{\mu_n}(u_n)u_n = \displaystyle\int_{\Omega}\gamma\left(\frac{u_n^2 + \lvert \nabla u_n \rvert^2}{2}\right)(u_n^2 + \lvert \nabla u_n\rvert^2)\, dx  - \mu_n \displaystyle\int_{\Omega} f_1((u_n)_+)u_n \, dx$$ $$+ \displaystyle\int_{\Omega}f_2((u_n)_+)u_n\, dx  - \displaystyle\int_{\Omega}hu_n\, dx = 0.$$
Thus, using $(q3)$,
\begin{equation}\label{l2}
\gamma(\infty)\displaystyle\int_{\Omega}\lvert \nabla u_n\rvert^2\, dx  \leq \mu_n \displaystyle\int_{\Omega} f_1((u_n)_+)u_n\, dx -
\displaystyle\int_{\Omega}f_2((u_n)_+)u_n\, dx  + \displaystyle\int_{\Omega}hu_n\, dx.
\end{equation}
Observe that if we know that $(u_n)_{n \in \NN}$  is bounded in $L^2(\Omega)$, \eqref{l2} in combination with the linear growth of $f$ imply that $(u_n)_{n \in \NN}$  is bounded in $H^1_0(\Omega)$, too. Let us thus prove that $(u_n)_{n \in \NN}$ is bounded in $L^2(\Omega).$  Arguing by contradiction, we assume that $\|u_n\|_2\ri\infty$. Set $v_n=u_n/\|u_n\|_2$, hence $\|v_n\|_2=1$. Let us show  that $(v_n)_{n \in \NN}$ is bounded in $H^1_0(\Omega)$. Dividing $\mathcal{E}'_{\mu_n}(u_n)u_n = 0$ by $||u_n||_2^2$ we get
$$\begin{array}{ll} \label{l3}
 \di\intom\gamma\left(\frac{u_n^2+|\nabla u_n|^2}{2}\right) \big( v_n^2 + |\nabla v_n|^2 \big) \, dx & =  \mu_n \displaystyle\int_{\Omega} f_1((u_n)_+)\, \frac{u_n}{||u_n||_2^2}\, dx   \smallskip\\
& - \displaystyle\int_{\Omega} f_2((u_n)_+)\, \frac{u_n}{||u_n||_2^2}\, dx - \int_{\Omega} h \, \frac{u_n}{||u_n||_2^2} \, dx.
\end{array}$$
But, by $(f4)$, $\|\mu_n f_1((u_n)_+) - f_2((u_n)_+)\|_2 \leq C(1+\|u_n\|_2)$ and thus the right-hand side is bounded.  Since $\gamma(s) \geq \gamma(\infty) >0$ for any $s \in [0, \infty)$ we deduce that  $(v_n)_{n \in \NN}$ is bounded in $H_0^1(\Omega)$. So, up to a subsequence,
$$ v_n\rightharpoonup v \geq 0 \quad\mbox{in } H^1_0(\Omega) \quad \mbox{and} \quad v_n\ri v \geq 0\quad\mbox{in $L^2(\Omega)$}.$$
By (f3), there exists $A>\gamma(\infty)(1+\lambda_1)$ and $B>0$ such that
$$f(t)\geq A t-B\qquad\mbox{for all $t \in [0, \infty)$}.$$
Assuming that $n \in \NN$ is large enough we can assume, modifying $A$ is necessary, that
\begin{equation}\label{ajout4}
\mu_n f_1(t) - f_2(t) \geq A t-B \quad \mbox{for all $t \in [0, \infty)$}.
\end{equation}
Now, from the fact that $ \displaystyle \mathcal{E}'_{\mu_n}(u_n) \frac{\phi_1}{||u_n||_2} = 0$, we get using \eqref{ajout4},
\begin{align}\label{ajout5}
\di\intom\gamma\left(\frac{u_n^2+|\nabla u_n|^2}{2}\right) \Big( \nabla v_n\cdot\nabla\varphi_1 + v_n\varphi_1\Big) \, dx & \geq \di A\intom v_n\varphi_1dx
-\frac{B}{\|u_n\|_2}\, \intom\varphi_1 \, dx\nonumber\\
& -  \frac{1}{\|u_n\|_2}\, \intom h \varphi_1 \, dx.
\end{align}
Using Proposition \ref{lappendice}, which will be proved in the Appendix,  the left-hand side of \eqref{ajout5} tends to $$\gamma(\infty)\intom (v\varphi_1+\nabla v\cdot\nabla\varphi_1) \, dx=\gamma(\infty)(1+\lambda_1)\intom v\varphi_1 \, dx.$$ Thus, we get from \eqref{ajout5}, as $n\ri\infty$,
$$ \gamma(\infty)(1+\lambda_1)\intom v\varphi_1 \, dx\geq A\intom v\varphi_1 \, dx,$$
which contradict the assumption $A>\gamma(\infty)(1+\lambda_1)$ since $ v \geq 0$ and $v \not \equiv 0$.
\end{proof}

\begin{prop}\label{sequenceboundedPS}
The sequence $(u_n)_{n \in \NN}$ is a (bounded) Palais-Smale sequence for $\mathcal{E}(u)$.
\end{prop}

\begin{proof}
Since the map $\mu \mapsto c_{\mu}$ is continuous from the left, see \cite[Lemma 2.3]{Je} for the proof, we get
\begin{itemize}
\item $\mathcal{E}(u_n) = \mathcal{E}_{\mu_n}(u_n) + (\mu_n - 1)B(u_n) \underset{n \rightarrow \infty}{\longrightarrow} \left(\underset{n \rightarrow \infty}{\text{lim}} c_{\mu_n} = c\right)$
\item $\mathcal{E}'(u_n) = \mathcal{E}'_{\mu_n}(u_n) + (\mu_n - 1)B'(u_n) \underset{n \rightarrow \infty}{\longrightarrow}0$ in the dual $H_0^{-1}\left(\Omega\right).$
\end{itemize}
Namely, $\left(u_n\right)_{n \in \mathbb{N}}$ is a (bounded) Palais-Smale sequence for $\mathcal{E}$ at the mountain pass level $c$.
\end{proof}

At this point we can give,

\begin{proof}[Proof of Theorem \ref{thm2}]
Part (i) is a direct consequence of Lemma \ref{firstsolution}. To prove Part (ii) it suffices to observe that by Lemma \ref{sequencebounded} the bounded Palais-Smale sequence $(u_n)_{n \in \NN}$ strongly converge.
\end{proof}

\section{Appendix}

\begin{prop}\label{lappendice}
In the setting of Lemma \ref{sequencebounded}, we have, $\forall w \in H_0^1\left(\Omega\right),$
\begin{equation}\label{convergence}
\displaystyle\int_{\Omega}\gamma\left(\frac{u_n^2 + \lvert \nabla u_n \rvert^2}{2}\right)(v_n w + \nabla v_n \cdot \nabla w)\, dx \underset{n \rightarrow \infty}{\longrightarrow} \gamma(\infty)\displaystyle\int_{\Omega} (v w + \nabla v \cdot \nabla w)\, dx.
\end{equation}
\end{prop}
\begin{proof}
We follow closely the proof of \cite[Lemma 7.2]{stuart}. Setting  $t_n =  \frac{1}{||u_n||_2}$ 	and using the weak convergence of
$v_n \rightharpoonup v$ in $H_0^1(\Omega)$, we see that \eqref{convergence} is equivalent to
$$\displaystyle\int_{\Omega}\left(\gamma\left(\frac{v_n^2 + \lvert \nabla v_n \rvert^2}{2t_n^2}\right) - \gamma(\infty)\right)(v_n w + \nabla v_n \cdot \nabla w)\, dx \underset{n \rightarrow \infty}{\longrightarrow} 0.$$
 Let $z_n = (v_n,\nabla v_n)$ and $y = (w,\nabla w)$. Then, $z_n, y  \in L^2\left(\Omega\right)^{N + 1}$ and, since  $(v_n)_{n \in \mathbb{N}}$ is bounded in $H_0^1\left(\Omega\right)$, there exists a constant $Z >0$ such that $\int_{\Omega}|z_n|^2 dx \leq Z$ for all $n \in \NN$.  \smallskip

Also, setting for $(m,n) \in \mathbb{N}^2$, $A^m_n = \left\{x \in \Omega / \lvert z_n \rvert^2 \leq \frac{1}{m}\right\}$ and $B^m_n = \Omega \backslash A^m_n$, we have,
$$\displaystyle\int_{\Omega}\left(\gamma\left(\frac{v_n^2 + \lvert \nabla v_n \rvert^2}{2t_n^2}\right) - \gamma(\infty)\right)(v_n w + \nabla v_n \nabla w)\, dx =
 \displaystyle\int_{\Omega}\left(\gamma\left(\frac{\lvert z_n\rvert^2}{2t_n^2}\right) - \gamma(\infty)\right)\langle z_n,y\rangle\, dx$$
$$= \left(\displaystyle\int_{A^m_n} + \displaystyle\int_{B^m_n}\right)\left[\left(\gamma\left(\frac{\lvert z_n\rvert^2}{2t_n^2}\right) - \gamma(\infty)\right)\langle z_n,y\rangle\, dx\right].$$
For all $(m,n) \in \mathbb{N}^2$, since $\gamma(s)$ is bounded by $(q3)$ there exists $C > 0$ such that,
$$\displaystyle\int_{A^m_n}\left\lvert \gamma\left(\frac{\lvert z_n\rvert^2}{2t_n^2}\right) - \gamma(\infty)\right\rvert  |z_n||y| \, dx \leq \frac{C}{\sqrt{m}}\displaystyle\int_{A^m_n}\lvert y \rvert\, dx \leq \left(C\sqrt{\frac{\lvert \Omega \rvert}{m}}\right)   \displaystyle \left(\int_{A^m_n}\lvert y \rvert\ ^2 \,dx \right)^{\frac{1}{2}}.$$
Also, for every $m \in \NN$, there exists $S_m >0$ such that $\lvert \gamma(s) - \gamma(\infty) \rvert < \frac{1}{m}$ for all $s \geq S_m$. Since $t_n \to 0$, there exists $N(m) >0$ such that $t_n^2 \leq  \displaystyle \frac{1}{2m S_m}$ for all $n \geq N(m)$. Hence $ \displaystyle |\gamma(\frac{1}{2t_n^2}|z_n|^2) - \gamma(\infty)| < \frac{1}{m} $ $ n \geq N(m)$ on $B^m_n$ and so
$$ \begin{array}{ll} \displaystyle\int_{B^m_n}\left\lvert \gamma\left(\frac{\lvert z_n\rvert^2}{2t_n^2}\right) - \gamma(\infty)\right\rvert  |z_n||y| \; dx & \displaystyle\leq \frac{1}{m}  \left( \int_{\Omega} |z_n|^2 dx \right)^{\frac{1}{2}}  \left( \int_{\Omega} |y|^2 dx \right)^{\frac{1}{2}}\\
&\displaystyle \leq \frac{\sqrt{Z}}{m}  \left( \int_{\Omega} |y|^2 dx \right)^{\frac{1}{2}}.\end{array}$$
Thus, it follows that, $\forall m \in \mathbb{N},\; \forall n \in \mathbb{N},\; n \geq N(m)$,
$$\displaystyle\int_{\Omega}\left\lvert \gamma\left(\frac{\lvert z_n\rvert^2}{2t_n^2}\right) - \gamma(\infty)\right\rvert  |z_n||y| \; dx \leq \left(C\sqrt{\frac{\lvert \Omega \rvert}{m}} + \frac{\sqrt{Z}}{m}\right) \left( \int_{\Omega} |y|^2 dx \right)^{\frac{1}{2}},$$
from which we deduce that
$$\displaystyle\int_{\Omega}\left\lvert \gamma\left(\frac{\lvert z_n\rvert^2}{2t_n^2}\right) - \gamma(\infty)\right\rvert  |z_n||y| \; dx \underset{n \rightarrow \infty}{\longrightarrow} 0,$$
ending the proof.
\end{proof}

\medskip
{\bf Acknowledgements.} This work was initiated during a visit of the second author to the University of Franche-Comt\'e in March 2013.  The second author thanks the University of Franche-Comt\'e for the financial support. The authors are grateful to Professor Haim Brezis for his valuable comments on a previous version of this paper. The first author also thanks M. L. M. Carvalho and E. D. Silva for discussing with him their results in \cite{CaGoSi,SiCaSiGo} connected with Lemma \ref{LEM3} and D.~Arcoya and J.~Giacomoni for useful remarks.

\end{document}